\documentclass[12pt]{amsart}
\usepackage{amsmath, amsfonts, amsbsy, amsthm, amscd, graphicx}
\usepackage{amssymb, latexsym, color}
\usepackage{bm}

\numberwithin{equation}{section}

\theoremstyle{plain}
\newtheorem{Theorem}{Theorem}[section]
\newtheorem{Proposition}[Theorem]{Proposition}

\theoremstyle{definition}
\newtheorem{Definition}[Theorem]{Definition}

\theoremstyle{remark}
\newtheorem{Remark}{{\bf Remark}}

\begin{document}
\title{Embedding and the rotational dimension of a graph containing a clique}

\author{Takumi Gomyou} 
\thanks{Graduate School of Mathematics, Nagoya University,
Chikusa-ku, Nagoya 464-8602, Japan, d16001m@math.nagoya-u.ac.jp}

\maketitle

\begin{abstract}
The rotational dimension is a minor monotone graph invariant related to the dimension of an Euclidean space containing a spectral embedding corresponding to the first nonzero eigenvalue of the graph Laplacian, which is introduced by G\"{o}ring, Helmberg and Wappler.
In this paper, we study rotational dimensions of graphs which contain large complete graphs.
The complete graph is characterized by its rotational dimension.
It will be a obtained that a chordal graph may be made large while keeping the rotational dimension constant.
\end{abstract}

\section{Introduction}
\label{intro}
Let $G=(V,E)$ be an undirected, simple, connected and finite graph, where $V = \{ 1, \cdots, n \}$ is the vertex set, $E = \{ ij \mid i \in V \ \mbox{is adjacent to} \ j \in V \}$ is the edge set.
Setting a vertex weight $s: V \to \mathbb{R}_{\geq 0}$ and an edge length $l: E \to \mathbb{R}_{\geq 0}$, G\"{o}ring, Helmberg and Wappler \cite{Goring2} introduced an optimization problem which is to maximize $\sum_{i \in V} s(i) \| \bm{v}(i) \|^2$ over all embedding map $\bm{v}$ subjected to the {\it equilibrium constraint} $\| \sum_{i \in V} s(i) \bm{v}(i) \|^2 = 0$ and the {\it distance constraints} $\| \bm{v}(i) -\bm{v}(j) \| \leq l(ij), \forall ij \in E$. 
This optimization problem is formulated as the following: 
\begin{equation}\label{emb_opt}
\left.
\begin{aligned}
  \mathrm{maximize} \quad & \sum_{i \in V} s(i) \| \bm{v}(i) \|^2 \\
  \mathrm{subject \ to} \quad & \| \sum_{i \in V} s(i) \bm{v}(i) \|^2 = 0, \\
  & \| \bm{v}(i) -\bm{v}(j) \| \leq l(ij) , \quad \forall ij \in E, \\
  & \bm{v}: V \to \mathbb{R}^n.
\end{aligned}
\right.
\end{equation}
The problem \eqref{emb_opt} is related to some maximization of the first nonzero eigenvalue of the graph Laplacian.
The eigenvalue optimization is formulated to a semidefinite programming, and then the problem \eqref{emb_opt} is derived from the dual of that.  
The non-parameterized versions of the problems were initially studied in \cite{Goring}.
A map $\bm{v}: V \to \mathbb{R}^n$, in the problem \eqref{emb_opt}, is considered as an embedding of vertices in $n$ (the cardinality of the vertex set)-dimensional Euclidean space.
However, the dimension of an Euclidean space containing the image of an optimal embedding does not necessarily have to be full.
It is interesting to find the minimal dimension of optimal embeddings. 
The rotational dimension is an invariant of a graph defined as follows. 
\begin{Definition}[rotational dimension \cite{Goring2}]
\begin{align*}
\mathrm{rotdim}_G (s,l) &:= \mathrm{min} \{ \mathrm{dim \ span} \{ \bm{v}(i) \in \mathbb{R}^n \mid i \in V \} \mid \bm{v} \ \text{is optimal of} \, \eqref{emb_opt} \}, \\
\mathrm{rotdim} (G) &:= \mathrm{max} \{ \mathrm{rotdim}_G (s,l) \mid s: V \to \mathbb{N} \cup \{ 0 \}, l: E \to \mathbb{N} \cup \{ 0 \} \}.
\end{align*}
Here, $\mathrm{dim \ span} \{ \bm{v}(i) \in \mathbb{R}^n \mid i \in V \}$ is the dimension of the linear subspace spanned by $\bm{v}(1), \cdots , \bm{v}(n)$.
$\mathrm{rotdim} (G)$ is called the {\it rotational dimension} of $G$.
\end{Definition}
G\"{o}ring, Helmberg and Wappler defined the rotational dimension and showed that $s: V \to \mathbb{N} \cup \{ 0 \}$, $l: E \to \mathbb{N} \cup \{ 0 \}$ can be replaced with $s: V \to \mathbb{R}_{\geq 0}$, $l: E \to \mathbb{R}_{\geq 0}$ or $s: V \to \mathbb{R}_{>0}$, $l: E \to \mathbb{R}_{>0}$ in the definition of $\mathrm{rotdim} (G)$ \cite{Goring2}.
The rotational dimension is a minor-monotone graph invariant. 
Several other minor-monotone invariants or optimal embeddings related to the first nonzero eigenvalue of the Laplacian are also known, such as the Colin de Verdi\`{e}re number \cite{Holst2} or the valid representation \cite{Holst}.

We find optimal embeddings and rotational dimensions of graphs which contain large complete graphs.
For any graph $G$ on $n$ vertices we have $\mathrm{rotdim} (G) \leq n-1$ because an optimal embedding satisfies the equilibrium constraint.
The complete graph $K_n$ attains $\mathrm{rotdim} (G) = n-1$.
The main results is the characterization of the complete graph by its rotational dimension, that is, a graph $G$ on $n$ vertices is complete if and only if $\mathrm{rotdim} (G) = n-1$.
A similar result is known for the valid representation invariant \cite{Holst}.
To show the property, we determine the rotational dimension of a graph obtained by removing one edge from the complete graph $K_n$.
When the parameters $s$ and $l$ are set uniformly, an optimal embedding of this graph is found uniquely.
The vertices of a complete subgraph $K_{n-1}$ are mapped bijectively onto the vertices of an $(n-2)$-regular simplex.
Also for an arbitrary graph, complete subgraph $K_m$ is embedded similarly under a certain natural assumption. 

For any graph the rotational dimension is bounded above by the tree-width plus one \cite{Goring2}, and bounded below by the clique number minus one. 
If a graph is chordal, the bounds of the rotational dimension are tight.
Applying the properties of a chordal graph, we may make a chordal graph larger while keeping the rotational dimension constant.
For instance, we determine the rotational dimension of a graph that consists of 
a complete graph $K_m$ and the other $k$ vertices $m+1, \cdots , m+k$ such that each vertex of $K_m$ connects completely $m+1, \cdots , m+k$.

This paper is organized as follows.
In Section 2 we recall G\"{o}ring, Helmberg and Wappler's work on the optimization problems and the rotational dimension. 
In Section 3 we find an optimal embedding and the rotational dimension for a graph obtained by removing one edge from a complete graph.
In Section 4 we study the rotational dimension of a chordal graph.

\section{Optimization problems and rotational dimension}

In this section, we recall the work of G\"{o}ring, Helmberg and Wappler \cite{Goring2} (also \cite{Goring}).
They derived the graph embedding problem \eqref{emb_opt} from the maximization of the first nonzero eigenvalue of the Laplacian via the semidefinite duality \cite{BoydVandenberghe}.
It can be confirmed that the spectral embedding which corresponds to the first nonzero eigenvalue is an optimal solution to \eqref{emb_opt}.
Then the dimension of an optimal embedding is bounded by the multiplicity of the first nonzero eigenvalue.
G\"{o}ring, et al. also gave the tree-width bound, which is applied to the rotational dimension, and presented some other properties of the rotational dimension.
Throughout this section, except where mentioned, vertex weight and edge length are all positive, that is, $s: V \to \mathbb{R}_{>0}$ and $l: E \to \mathbb{R}_{>0}$.

\subsection{Duality between maximization of the first nonzero eigenvalue and graph embedding problem}
We take an edge weight $w: E \to \mathbb{R}_{\geq 0}$.
Let $E_{ij}$ be an $n \times n$ symmetric matrix whose $ii$ and $jj$ components are $1$, $ij$ and $ji$ components are $-1$, and all other components are zero.
With a diagonal matrix $D = \mathrm{diag}(s(1)^{-1/2}, \cdots, s(n)^{-1/2})$, we consider the Laplacian 
$$L(G;(w,s)) := D \left( \sum_{ij \in E} w(ij) E_{ij} \right) D$$
for a graph $G=(V,E)$. 
The Laplacian $L(G;(w,s))$ is positive semidefinite (denoted as $L(G;(w,s)) \succeq 0$), 
and it has exactly one eigenvalue $0$ if and only if the graph $G_w = (V, E_w := \{ ij \in E \mid w(ij)>0 \})$ is connected. 
The starting optimization problem is the maximization of the first nonzero eigenvalue $\lambda_1 (L(G;(w,s)))$: 
\begin{equation}\label{eigen_opt}
\left.
\begin{aligned}
  \widehat{a}(G;s,l) := \mathrm{maximize} \quad & \lambda_1 (L(G;(w,s))) \\
  \mathrm{subject \ to} \quad & \sum_{ij \in E} l(ij)^2 w(ij) \leq 1, \\
  & w: E \to \mathbb{R}_{\geq 0}.
\end{aligned}
\right.
\end{equation}
\begin{Remark}
In this optimization, an optimal solution $w^*$ lies on the boundary of the inequality  constraint $\sum_{ij \in E} l(ij)^2 w(ij) \leq 1$. 
Therefore, it can be replaced with the equality constraint.
When $s$ and $l$ are uniformly one: $s = {\bf 1}$ and $l = {\bf 1}$, the scaled problem 
\begin{equation}
\left.
\begin{aligned}
  |E| \, \widehat{a}(G;{\bf 1},{\bf 1}) = \mathrm{maximize} \quad & \lambda_1 (L(G;(w,{\bf 1}))) \\
  \mathrm{subject \ to} \quad & \sum_{ij \in E} w(ij) = |E|, \\
  & w: E \to \mathbb{R}_{\geq 0}
\end{aligned}
\right.
\end{equation}
was introduced by Fiedler \cite{Fiedler}.
He called the optimal value $\widehat{a} (G) := |E| \, \widehat{a}(G;{\bf 1},{\bf 1})$ the {\it absolute algebraic connectivity} of $G$.
\end{Remark}
The problem \eqref{eigen_opt} is rewritten in a standard form of a semidefinite programming:
\begin{equation}\label{eigen_opt2}
\left.
\begin{aligned}
  \frac{1}{\widehat{a}(G;s,l)} = \mathrm{minimize} \quad & \sum_{ij \in E} l(ij)^2 \widetilde{w}(ij) \\
  \mathrm{subject \ to} \quad & L(G;(\widetilde{w},s)) +\mu D^{-1} {\bf 1} \, {}^{t}{\bf 1} D^{-1} -I \succeq 0, \\
  & \widetilde{w}: E \to \mathbb{R}_{\geq 0}, \\
  & \mu \in \mathbb{R}.
\end{aligned}
\right.
\end{equation}
Here $\widetilde{w}$ is a scaled edge weight such that 
$$\widetilde{w} = \frac{w}{\lambda_1 (L(G;(w,s)))}.$$
The embedding problem \eqref{emb_opt} is formulated form the dual of \eqref{eigen_opt2}. 
There is no duality gap for this primal-dual pair, that is, the optimal value of \eqref{emb_opt} is equal to $1/\widehat{a}(G;s,l)$.
Now we assume that $\bm{v}$ and $w$ satisfy the constraints, respectively.
(An optimization variable satisfying the constraints is called a {\it feasible solution}.)
These are optimal respectively, if and only if, these satisfy
\begin{align}
& w(ij) \left( l(ij)^2 -\| \bm{v}(i) -\bm{v}(j) \|^2 \right) =0, \quad \forall ij \in E \quad \mathrm{and} \label{KKT1} \\ 
& \sum_{k:ik \in E} w(ik) \left( \bm{v}(i) -\bm{v}(k) \right) = \lambda_1 (L(G;(w,s))) \, s(i) \bm{v}(i), \quad \forall i \in V. \label{KKT2}
\end{align}
The conditions \eqref{KKT1} and \eqref{KKT2} are obtained from the {\it complementary slackness conditions}.
The second formula \eqref{KKT2} is equivalent to
\begin{equation*}
L(G;(w,s)) D^{-1} 
\begin{pmatrix}
{}^t \bm{v}(1) \\ 
\vdots \\
{}^t \bm{v}(n) 
\end{pmatrix}
= \lambda_1 (L(G;(w,s))) D^{-1} 
\begin{pmatrix}
{}^t \bm{v}(1) \\
\vdots \\
{}^t \bm{v}(n) 
\end{pmatrix}
.
\end{equation*}
This means that each column of $D^{-1} \, {}^{t} \left( \bm{v}(1) \cdots \bm{v}(n) \right)$ is an eigenvector corresponding to $\lambda_1 (L(G;(w,s)))$. 
Thus, the dimension of the linear subspace spanned by $\bm{v}(1), \cdots , \bm{v}(n)$ are less than the multiplicity of $\lambda_1 (L(G;(w,s)))$.
Especially, the rotational dimension is bounded above by the multiplicity of $\lambda_1 (L(G;(w,s)))$.

\subsection{Properties of rotational dimension}

In this subsection, we summarize the results of \cite{Goring, Goring2} for an optimal embedding and the rotational dimension.

The structural property of an optimal embedding is closely associated with a separator of a graph. 
\begin{Theorem}[Separator-Shadow \cite{Goring2}]\label{Separator-Shadow}
For a graph $G=(V,E)$ with parameters $s: V \to \mathbb{R}_{>0}$ and $l: E \to \mathbb{R}_{>0}$, let $\bm{v}: V \to \mathbb{R}^n$ and $w: E \to \mathbb{R}_{\geq 0}$ be optimal solutions, respectively. 
And let $S$ be a separator that divides the graph $G_w = (V, E_w)$ into two connected components $C_1 \subset G$ and $C_2 \subset G$.
Then for at least one of $C_1$ and $C_2$, say $C_1$,  
$$\mathrm{conv} \{ 0, \bm{v}(i) \} \cap \mathrm{conv} \{ \bm{v}(s) \mid s \in S \} \neq \emptyset , \quad \forall i \in V(C_1).$$
Here, $\mathrm{conv} \{ 0, \bm{v}(i) \}$ is a line segment connecting the origin and $\bm{v}(i)$, and $\mathrm{conv} \{ \bm{v}(s) \mid s \in S \}$ is the convex hull of the set $\{ \bm{v}(s) \mid s \in S \}$.
\end{Theorem}
Note that any separator of a graph $G=(V,E)$ is also a separator of $G_w = (V, E_w)$.
If we regard the origin as a light source and the convex hull of the separator's points as a solid body, then all the vertices in at least one component are mapped at the shadow of the separator.
Thus, this theorem was named Separator-Shadow Theorem.
By this theorem, the dimension of the subspace spanned by $S$ and one component which not contained in the shadow of $S$ attains the dimension for the whole vertex.

One of main results by G\"{o}ring, et al. is the tree-width bound on the rotational dimension.
To state that, we recall the definitions of a tree-decomposition and the tree-width.
\begin{Definition}
A {\it tree-decomposition} of a graph $G = (V,E)$ is a tree $T$ whose vertex set $V(T)$ is a family of subsets of $V$, satisfying the following properties.
\begin{itemize}
\item[(i)] $V= \underset{U \in V(T)}{\cup} U$, 
\item[(ii)] for any edge $ij \in E$ there exists $U \in V(T)$ such that $i,j \in U$, 
\item[(iii)] if $U_1 , U_2 , U_3 \in V(T)$ and the path between $U_1$ and $U_2$ contains $U_3$, then $U_3 \supset U_1 \cap U_2$.
\end{itemize}
The {\it width} of a tree-decomposition is the cardinality of the largest size subset of $V$ minus $1$, and the {\it tree-width} $\mathrm{tw} (G)$ of a graph $G$ is the minimum width over all tree-decompositions of $G$.
\end{Definition}
\begin{Theorem}[\cite{Goring2}]\label{rotdim-twbound} 
For any graph $G$ and any parameters $s$, $l$ there exists an optimal embedding whose dimension is less than or equal to the tree-width of $G$ plus one.
\end{Theorem}
It is clearly to see that the rotational dimension has the same bound.

Finally, we recall some facts about the rotational dimension.
The rotational dimension is monotone under taking minors. 
\begin{Definition}
A graph $G'$ obtained by repeating three operations $(i)$ deletion of isolated vertex, \ $(ii)$ deletion of edge, \ $(iii)$ contraction of edge from $G$, is called a {\it minor} of $G$.
Then we write $G \succeq G'$.
\end{Definition}
\begin{Theorem}[minor monotonicity of rotational dimension \cite{Goring2}]
\label{minor rotdim}
$$G \succeq G' \Rightarrow  \mathrm{rotdim} (G) \geq \mathrm{rotdim} (G').$$
\end{Theorem}
When a graph $G$ is disconnected, the rotational dimension of $G$ is defined as
$$\mathrm{rotdim} (G) := \mathrm{max} \{ \mathrm{rotdim} (C) \mid C \ \text{is a connected component of} \ G \} .$$
Graphs having low rotational dimensions are classified as follows.
\begin{Theorem}[\cite{Goring2}]
\begin{align*}
& \mathrm{rotdim} (G) = 0 \Leftrightarrow G \ does \ not \ have \ edges,\\
& \mathrm{rotdim} (G) \leq 1 \Leftrightarrow G \ is \ composed \ of \ disjoint \ union \ of \ paths,\\
& \mathrm{rotdim} (G) \leq 2 \Leftrightarrow G \ is \ outer \ planar.
\end{align*}
\end{Theorem}

\section{Rotational dimension of a complete graph}

In this section, we find an optimal embedding and the rotational dimension of a graph obtained by removing one edge from a complete graph.
By determining the rotational dimension of this graph, we obtain the characterization of the complete graph by its rotational dimension.
First, an upper bound on the rotational dimension for any graph on $n$ vertices is obtained as follows.
\begin{Proposition} \label{rotdim1}
If $G$ is a graph on $n$ vertices, then
$$\mathrm{rotdim} (G) \leq n-1.$$ 
\end{Proposition}
\begin{proof}
For any parameters $s: V \to \mathbb{R}_{>0}$, $l: E \setminus \{ e \} \to \mathbb{R}_{>0}$ and any optimal embedding $\bm{v}: V \to \mathbb{R}^n$, we have $\mathrm{dim \ span} \{ \bm{v}(i) \in \mathbb{R}^n \mid i \in V \} \leq n-1$, because the equilibrium constraint $\| \sum_{i \in V} s(i) \bm{v}(i) \|^2 = 0$ is satisfied.
Therefore, $\mathrm{rotdim} (G) \leq n-1$.
\end{proof}
When the parameters $s$, $l$ are both uniform, a regular simplex is a unique optimal embedding for a complete graph.
Here, we regard embeddings which differ by a rotation around the origin as same. 
Then we have $\mathrm{rotdim}_{K_n} (1,1) = n-1$.
The rotational dimension $\mathrm{rotdim} (K_n)$ has same value.
\begin{Proposition}\label{rotdim-completegraph}
$$\mathrm{rotdim} (K_n) = n-1.$$
\end{Proposition}
\begin{proof}
We have $\mathrm{rotdim} (K_n) \geq \mathrm{rotdim}_{K_n} (1,1) = n-1$ by the definition of the rotational dimension.
The reverse inequality is obtained by Proposition \ref{rotdim1}.
\end{proof}
We can immediately see that $\mathrm{rotdim} (G) \geq \omega (G) -1$ by this proposition and Theorem \ref{minor rotdim}.
where $\omega (G)$ is the clique number of $G$.
A complete subgraph is called a {\it clique}, and a {\it clique number} $\omega (G)$ is the number of vertices of the largest size clique in $G$.

The following theorem is the main result in this paper. We characterize the complete graph by its rotational dimension.
\begin{Theorem} \label{complete rotdim}
If $G$ is a graph on $n$ vertices, then
$$\mathrm{rotdim} (G) = n-1 \ \Leftrightarrow \ G=K_n.$$
\end{Theorem}
A similar property for the valid representation invariant $\lambda (G)$ is obtained by Holst-Laurent-Schrijver \cite{Holst}.
In order to prove Theorem \ref{complete rotdim}, we consider a graph obtained by removing one edge from the complete graph $K_n$, which is the clique sum of two $K_{n-1}$'s.
\begin{Definition}
Let $G_1$ and $G_2$ be graphs with same size cliques. 
The {\em clique sum} $G_1 \oplus G_2$ of $G_1$ and $G_2$ is the graph obtained by identifying the respective cliques in the disjoint union of $G_1$ and $G_2$.
\end{Definition}
The rotational dimension of the graph $K_n \setminus \{ e \}$ is calculated as follows.
\begin{Theorem} \label{edge complete}
Let $K_n = (V,E)$ be a complete graph on $n$ vertices.
We take one edge $e \in E$. 
Then
$$\mathrm{rotdim} (K_n \setminus \{ e \} ) = n-2.$$
\end{Theorem}
\begin{proof}
Since $K_n \setminus \{ e \}$ contains $K_{n-1}$ as a minor, $\mathrm{rotdim} (K_n \setminus \{ e \} ) \geq n-2$ holds by Theorem \ref{minor rotdim}.
We prove the other inequality.
Let $\bm{v}: V \to \mathbb{R}^n$ be an optimal embedding for $K_n \setminus \{ e \}$ that attains $\mathrm{rotdim}_{K_n \setminus \{ e \}} (s, l)$, for $s: V \to \mathbb{R}_{>0}$ and $l: E \setminus \{ e \} \to \mathbb{R}_{>0}$. 
If $\mathrm{rotdim}_{K_n \setminus \{ e \}} (s, l) \leq n-2$ is shown for such parameters, then we can obtain $\mathrm{rotdim} (K_n \setminus \{ e \} ) \leq n-2$.
We regard $K_n \setminus \{ e \}$ as $K_{n-1} \oplus K_{n-1}$ with a common clique $K_{n-2}$.
Let $S := V \setminus \{ 1,2 \}$ be a vertex set of the common clique $K_{n-2}$.

Case (i) : $0 \not\in \mathrm{conv} \{ \bm{v}(i) \in \mathbb{R}^n \mid i \in S \}$. \ 
The dimension of the linear subspace $\mathrm{span} \{ \bm{v}(i) \in \mathbb{R}^n \mid i \in S \}$ is $n-2$ or less. 
By Theorem \ref{Separator-Shadow} we see that $\bm{v}(1) \in \mathrm{span} \{ \bm{v}(i) \in \mathbb{R}^n \mid i \in S \}$ or $\bm{v}(2) \in \mathrm{span} \{ \bm{v}(i) \in \mathbb{R}^n \mid i \in S \}$, say the former.
We also have $\bm{v}(2) \in \mathrm{span} \{ \bm{v}(i) \in \mathbb{R}^n \mid i \in S \}$ by the equilibrium constraint $\| \sum_{i \in V} s(i) \bm{v}(i) \|^2 = 0$.
Thus, $\mathrm{span} \{ \bm{v}(i) \in \mathbb{R}^n \mid i \in V \} \subset \mathrm{span} \{ \bm{v}(i) \in \mathbb{R}^n \mid i \in S \}$, and therefore, $\mathrm{dim \ span} \{ \bm{v}(i) \in \mathbb{R}^n \mid i \in V \} \leq n-2$.

Case (ii) : $0 \in \mathrm{conv} \{ \bm{v}(i) \in \mathbb{R}^n \mid i \in S \}$. \ 
Since the affine subspace $\text{aff-span} \{ \bm{v}(i) \in \mathbb{R}^n \mid i \in S \cup \{ 1 \} \}$ contains the origin, $\mathrm{dim \ span} \{ \bm{v}(i) \in \mathbb{R}^n \mid i \in S \cup \{ 1 \} \} $ is $n-2$ or less.
By the equilibrium constraint we have $v_2 \in \mathrm{span} \{ \bm{v}(i) \in \mathbb{R}^n \mid i \in S \cup \{ 1 \} \}$. 
Therefore, $\mathrm{dim \ span} \{ \bm{v}(i) \in \mathbb{R}^n \mid i \in V \} \leq n-2$.
\end{proof}
\noindent
{\it Proof of Theorem \ref{complete rotdim}.} \ 
We verify that if $G$ is a non-complete graph on $n$ vertices, then $\mathrm{rotdim} (G) \leq n-2$.
Such a graph is a minor of $K_n \setminus \{ e \}$.
Then, by Theorems \ref{minor rotdim} and \ref{edge complete}, we have $\mathrm{rotdim} (G) \leq \mathrm{rotdim} (K_n \setminus \{ e \} ) = n-2$.
\hfill $\qed$ \\ 
\\
\indent
Even if the rotational dimension can be calculated, it is generally difficult to find an optimal embedding that attains the rotational dimension.
For $K_n$ or $K_n \setminus \{ e \}$, however, such an embedding can be found when $s = \bm{1}$ and $l = \bm{1}$. 
The $(n-1)$-regular simplex is the only optimal embedding of $K_n$.
For a general graph its complete subgraph may be embedded similarly.
\begin{Proposition} \label{prop_completesubgraph}
Let $G=(V,E)$ be a graph on $n$ vertices with $s = \bm{1}$ and $l = \bm{1}$. 
For an optimal embedding $\bm{v}: V \to \mathbb{R}^n$ and a complete subgraph $K_m$ of $G$, 
if $\| \bm{v}(i) - \bm{v}(j) \| = 1$ holds for any $ij \in E(K_m)$, then the vertices of $K_m$ are mapped bijectively onto the vertices of the $(m-1)$-regular simplex with side length $1$.
\end{Proposition}

\begin{proof}
The statement is clearly satisfied when $m$ is $1$ or $2$.
Assume that a clique $K_{m-1}$ is embedded as a $(m-2)$-regular simplex with side length $1$ for $m \geq 3$.
We consider a new vertex $i'$ adjacent to all of vertices of $K_{m-1}$.
The vertex $i'$ is mapped on a straight line that is orthogonal to the given $(m-2)$-regular simplex and passes through the center of mass of the simplex, such that $\| \bm{v}(i') - \bm{v}(j) \| = 1$ for any $j \in V(K_{m-1})$.
As a result, we obtain a $(m-1)$-regular simplex with side length $1$.
\end{proof}

\begin{Remark}
When $s = \bm{1}$ and $l = \bm{1}$, if we find an optimal weight whose components are all nonzero, then the condition \eqref{KKT1} gives 
$$\| \bm{v}(i) -\bm{v}(j) \| = 1, \quad \forall ij \in E.$$
In many cases, this condition may be satisfied, thus Proposition \ref{prop_completesubgraph} is applied. 
However, a graph $K_3 \oplus K_3 \oplus K_3$ with the common clique $K_2$, in which each vertex of the common clique $K_2$ connects completely the other three vertices $\{ 1,2,3 \}$, can be embedded optimally by a map $\bm{v}: V \to \mathbb{R}^2$ such that
\begin{equation*}
\bm{v} (i) = 
\renewcommand{\arraystretch}{1.7}
\begin{cases}
{}^{t} (0, 0), & \text{if} \ i \in V(K_2), \\ 
{}^{t} ( \cos \frac{2}{3} \pi i, \sin \frac{2}{3} \pi i ), & \text{if} \ i= 1, 2, 3,
\end{cases}
\end{equation*} 
that is, the edge of $K_2$ is contracted to one point.
On the other hand, the edge weight 
\begin{equation*}
w(ij) = 
\renewcommand{\arraystretch}{1.7}
\begin{cases}
0, & \text{if} \ i \in V(K_2), \\ 
\frac{1}{6}, & \text{if} \ i= 1, 2, 3
\end{cases}
\end{equation*} 
is optimal. 
(See Proposition \ref{prop_sol_G(m,k)}.)
\end{Remark}
Applying Proposition \ref{prop_completesubgraph} to $K_n \setminus \{ e \}$, an optimal embedding can be found.
\begin{Proposition}\label{prop_sol_Kn-e}
The graph $K_n \setminus \{ e \}$ is considered as the clique sum $K_{n-1} \oplus K_{n-1}$ with the common clique $K_m$, where $m:=n-2 > 2$.
Denote the vertex set of the common clique $K_m$ by $\{ 1, \cdots , m \}$, and denote the other vertices by $m+1$, $m+2$. 
If $s = \bm{1}$ and $l = \bm{1}$, then an optimal embedding $\bm{v}: V \to \mathbb{R}^{n-2}$ is uniquely obtained as follows:
\begin{itemize}
\item[(i)] the vertices of $K_{m}$ are mapped bijectively onto the vertices of the $(m-1)$-regular simplex $\Delta$ with side length one whose barycenter is at the origin, 
\item[(ii)] the remaining vertices $m+1$ and $m+2$ are mapped on the straight line that is orthogonal to $\Delta$ and passes through the origin, such that the distances from each vertex of $\Delta$ to $\bm{v} (m+1)$ and $\bm{v} (m+2)$ are both one and they are centrally symmetric. 
\end{itemize}
On the other hand, the edge weight  
\begin{equation*}
w(ij) = 
\renewcommand{\arraystretch}{1.7}
\begin{cases} 
\frac{2(m-2)}{m(m^2+m+2)}, & \text{if} \ 1 \leq i,j \leq m,\\
\frac{2}{m^2+m+2}, & \text{if} \ 1 \leq i \leq m, \ j = m+1 \ \text{or} \ m+2
\end{cases}
\end{equation*}
attains 
\begin{equation*}
\widehat{a}(G;{\bf 1},{\bf 1}) = \frac{2m}{m^2+m+2}.
\end{equation*}
\end{Proposition}

\begin{proof}
The embedding $\bm{v}: V \to \mathbb{R}^n$ as (i) and (ii) is obviously feasible, in particular, $\| \bm{v}(i) -\bm{v}(j) \|^2 = 1$ is satisfied for any edge $ij$.
Hence, if the condition \eqref{KKT2} holds for $\bm{v}$ and some feasible weight, these variables are optimal respectively.
By taking the symmetry of the graph into account, let
\begin{equation*}
w(ij) = 
\begin{cases} 
a \geq 0, & \text{if} \ 1 \leq i,j \leq m,\\
b \geq 0, & \text{if} \ 1 \leq i \leq m, \ j = m+1 \ \text{or} \ m+2.
\end{cases}
\end{equation*}
be a feasible weight, that is, $w$ satisfies 
\begin{equation*}
\frac{m(m-1)}{2}a +2mb = 1.
\end{equation*} 
Then, by inserting $w$ and $\bm{v}$ into the condition \eqref{KKT2}:
\begin{equation*}
\sum_{k:ik \in E} w(ik) \left( \bm{v}(i) -\bm{v}(k) \right) = \lambda_1 (L(G;(w,\bm{1}))) \, \bm{v}(i), \quad \forall i \in V,
\end{equation*}
we have 
\begin{equation*}
a = \frac{2(m-2)}{m(m^2+m+2)}, \quad b = \frac{2}{m^2+m+2} 
\end{equation*}
Therefore, $\bm{v}$ and the obtained weight $w$ are optimal solutions, respectively. 
For this weight the first nonzero eigenvalue is
$$\lambda_1 (L(G;(w,\bm{1}))) = \frac{2m}{m^2+m+2},$$
which is optimal value of the problem \eqref{eigen_opt}.

We give an arbitrary optimal embedding $\bm{v}^*: V \to \mathbb{R}^n$.
By inserting $w$ and $\bm{v}^*$ into the conditions \eqref{KKT1} and \eqref{KKT2}, we obtain
\begin{align*}
\| \bm{v}^*(i) -\bm{v}^*(j) \|^2 = 1, \quad \forall ij \in E,\\
\sum_{1 \leq i \leq m} \bm{v}^*(i) = 0, \\
\bm{v}^*(m+1) +\bm{v}^*(m+2) = 0.
\end{align*}
Therefore, by combining these with Proposition \ref{prop_completesubgraph}, the embedding $\bm{v}^*$ is exactly (i) and (ii).
\end{proof}

\begin{Remark}
Since the embedding as (i) and (ii) is a unique optimal solution, we have
\begin{equation*}
\mathrm{rotdim}_{K_n \setminus \{ e \}} (\bm{1},\bm{1}) = n-2, \quad \text{if} \ n \geq 5.
\end{equation*}
Although the graph $K_4 \setminus \{ e \}$ can be also embedded as (i) and (ii), when $m(=n-2)=2$, an optimal weight having zero values does not bring the uniqueness of an optimal embedding. 
\end{Remark}

\section{Rotational dimension of a chordal graph}

By Theorems \ref{rotdim-twbound}, \ref{minor rotdim} and Proposition \ref{rotdim-completegraph}, we see that the rotational dimension is bounded by the tree-width $\text{tw} (G)$ or the clique number $\omega (G)$:
\begin{equation}\label{rotdimbound}
\omega (G) -1 \leq \mathrm{rotdim} (G) \leq \text{tw} (G)+1.
\end{equation}
These bounds are tight for a chordal graph. 
A {\it chordal graph} is a graph in which all cycles of length $4$ or more have a chord that is an edge connecting two nonadjacent vertices of this cycle.
For an arbitrary graph $G$ the inequality $\mathrm{tw} (G) \geq \omega (G) -1$ holds, however, if $G$ is a chordal graph, then it becomes the equality (see, e.g. \cite{Bodlaender}) and so \eqref{rotdimbound} becomes
$$\omega (G) -1 \leq \mathrm{rotdim} (G) \leq \omega (G).$$
Applying the properties of a chordal graph, we calculate the rotational dimension of some large graphs.
\begin{Proposition}\label{chordal_rotdim}
Let $G$ be a chordal graph that satisfies $\mathrm{rotdim} (G) = \omega (G)$. 
Also, let $\widehat{G}$ be a chordal graph containing $G$ as a subgraph. 
If $\omega (\widehat{G}) = \omega (G)$, then
$$\mathrm{rotdim} (\widehat{G}) = \mathrm{rotdim} (G).$$
\end{Proposition}

\begin{proof}
It is clearly to ensure that by the minor monotonicity of the rotational dimension.
\end{proof}
By using the technique in this proposition, rotational dimensions can be determined for graphs $K_{m+1} \oplus K_{m+1} \oplus \cdots \oplus K_{m+1}$ with the common clique $K_{m}$.
When $s = \bm{1}$ and $l = \bm{1}$, we observe that optimal embeddings and minimal dimensions.
\begin{Proposition}\label{prop_sol_G(m,k)}
Let $G(m,k)$ be the $k$-clique sum $K_{m+1} \oplus K_{m+1} \oplus \cdots \oplus K_{m+1}$ with the common clique $K_{m}$, that is, $G(m,k)$ is the graph of order $m+k$ in which $m$ vertices $1, \cdots , m$ form $K_m$ and each of the remaining $k$ vertices $m+1, \cdots , m+k$ is adjacent to all the vertices of $K_m$, where $m>k \geq 2$.
If $s = \bm{1}$ and $l = \bm{1}$, then an optimal embedding $\bm{v}: V \to \mathbb{R}^{m+1}$ is obtained as follows:
\begin{itemize}
\item[(i)] the vertices of $K_{m}$ are mapped bijectively onto the vertices of the $(m-1)$-regular simplex $\Delta$ with side length one whose barycenter is at the origin, 
\item[(ii)] the remaining vertices $m+1, \cdots , m+k$ are mapped on the circle with radius $r = \sqrt{(m+1)/2m}$ that is orthogonal to $\Delta$ (denote this circle as $\{ (x, y, 0, \cdots , 0) \in \mathbb{R}^{m+1} \mid x^2+y^2= r^2 \}$) as follows. \\
If $k$ is even, 
\begin{equation*}
\bm{v} (m+i) = 
\renewcommand{\arraystretch}{1.7}
\begin{cases}
(r, 0, \cdots , 0), & \text{if} \ i \ \text{is odd}, \\ 
(-r, 0, \cdots , 0), & \text{if} \ i \ \text{is even}. \\ 
\end{cases}
\end{equation*} 
If $k$ is odd, 
\begin{eqnarray*}
\bm{v} (m+1) &=& (r, 0, \cdots , 0), \\
\bm{v} (m+1+i) &=& 
\renewcommand{\arraystretch}{1.7}
\begin{cases}
(-\frac{r}{k-1}, \sqrt{r^2-(\frac{r}{k-1})^2}, 0, \cdots , 0), & \text{if} \ i  \ \text{is odd}, \\ 
(-\frac{r}{k-1}, -\sqrt{r^2-(\frac{r}{k-1})^2}, 0, \cdots , 0), & \text{if} \ i  \ \text{is even}. \\ 
\end{cases}
\end{eqnarray*} 
\end{itemize}
Moreover, this embedding attains 
\begin{equation*}
\mathrm{rotdim}_{G(m,k)} (\bm{1},\bm{1}) = 
\renewcommand{\arraystretch}{1.7}
\begin{cases}
m, & \text{if} \ k \ \text{is odd}, \\ 
m+1, & \text{if} \ k \ \text{is even}. \\ 
\end{cases}
\end{equation*}
\indent
On the other hand, the edge weight  
\begin{equation*}
w(ij) = 
\renewcommand{\arraystretch}{1.7}
\begin{cases} 
\frac{2(m-k)}{m(m^2+(k-1)m+k)}, & \text{if} \ 1 \leq i,j \leq m,\\
\frac{2}{m^2+(k-1)m+k}, & \text{if} \ 1 \leq i \leq m, \ m+1 \leq j \leq m+k
\end{cases}
\end{equation*}
attains
\begin{equation*}
\widehat{a}(G;{\bf 1},{\bf 1}) = \frac{2m}{m^2+(k-1)m+k}.
\end{equation*}
\end{Proposition}

\begin{proof}
For the embedding $\bm{v}: V \to \mathbb{R}^n$ as (i) and (ii), the condition \eqref{KKT1} holds because of  
$$\| \bm{v}(i) -\bm{v}(j) \| = 1 \quad \forall ij \in E.$$
By the discussion in a similar way to one in Proposition \ref{prop_sol_Kn-e}, we show that $\bm{v}$ is optimal, and obtain an optimal weight $w: E \to \mathbb{R}_{\geq 0}$ as 
\begin{equation*}
w(ij) = 
\renewcommand{\arraystretch}{1.7}
\begin{cases} 
\frac{2(m-k)}{m(m^2+(k-1)m+k)}, & \text{if} \ 1 \leq i,j \leq m,\\
\frac{2}{m^2+(k-1)m+k}, & \text{if} \ 1 \leq i \leq m, \ m+1 \leq j \leq m+k.
\end{cases}
\end{equation*}
\indent
By inserting the weight $w$ and any optimal embedding $\bm{v}^*$ into the conditions \eqref{KKT1} and \eqref{KKT2}, we obtain
\begin{align*}
\| \bm{v}^*(i) -\bm{v}^*(j) \|^2 = 1, \quad \forall ij \in E,\\
\sum_{1 \leq i \leq m} \bm{v}^*(i) = 0, \\
\sum_{m+1 \leq i \leq m+k} \bm{v}^*(i) = 0.
\end{align*}
Hence, it is shown that an optimal embedding as (i) and (ii) attains the minimal dimension of optimal embeddings, i.e., we obtain
\begin{equation*}
\mathrm{rotdim}_{G(m,k)} (\bm{1},\bm{1}) = 
\renewcommand{\arraystretch}{1.7}
\begin{cases}
m, & \text{if} \ k \ \text{is odd}, \\ 
m+1, & \text{if} \ k \ \text{is even}. 
\end{cases}
\end{equation*}
\end{proof}
By Proposition \ref{chordal_rotdim} and \ref{prop_sol_G(m,k)} the rotational dimension of $G(m,k)$ is determined as follows.
\begin{Theorem}
$$\mathrm{rotdim} (G(m,k)) = m+1, \quad \forall m \geq 4, \forall k \geq 3.$$
\end{Theorem}

\begin{proof}
The graph $G(m,3)$ is the chordal graph whose clique number is $m+1$.
Thus by Proposition \ref{prop_sol_G(m,k)}, 
$$m+1 = \mathrm{rotdim}_{G(m,3)} (\bm{1},\bm{1}) \leq \mathrm{rotdim} (G(m,3)) \leq m+1$$ 
holds, where $m \geq 4$.
Then we have 
$$\mathrm{rotdim} (G(m,3)) = m+1.$$
For any $k \geq 4$ the graph $G(m,k)$, which is also chordal, has $G(m,3)$ as a subgraph.
Since the clique number is constant, the proof is completed by Proposition \ref{chordal_rotdim}.
\end{proof}
Finally, note that $G(m,1) = K_{m+1}$ and $G(m,2) = K_{m+2} \setminus \{ e \}$, then we have
$$\mathrm{rotdim} (G(m,k)) = m, \quad \text{for} \ k = 1, 2.$$

\end{document}